\documentclass[12pt,reqno]{amsart}

\usepackage[T1]{fontenc}
\pdfoutput=1
\usepackage{mathptmx,microtype,graphicx}
\usepackage{amsthm,amsmath,amssymb}
\usepackage{enumitem}
\usepackage[font=footnotesize,margin=2.6cm]{caption}
\usepackage{cite}
\usepackage{xcolor}
\usepackage[hidelinks]{hyperref}
\usepackage{verbatim}
\makeatletter
\renewcommand\@biblabel[1]{#1.}
\makeatother

\theoremstyle{plain}
\newtheorem{thm}{Theorem}
\newtheorem{lemma}[thm]{Lemma}

\newtheorem{claim}[thm]{Claim}
\newtheorem{corollary}[thm]{Corollary}
\newtheorem{prop}[thm]{Proposition}

\newtheorem{question}[thm]{Question}
\theoremstyle{definition}
\newtheorem{definition}[thm]{Definition}

\theoremstyle{remark}

\newcommand\cG{\mathcal{G}}
\newcommand\Gnp{\mathcal{G}_{n,p}}
\newcommand\GnpH{\langle \mathcal{G}_{n,p}\rangle_H}

\newcommand\cA{\mathcal{A}}
\newcommand\cP{\mathcal{P}}
\newcommand\cH{\mathcal{H}}
\newcommand\cC{\mathcal{C}}

\newcommand\Z{\mathbb{Z}}

\renewcommand{\P}{{\mathbb P}}
\newcommand\E{{\mathbb E}}
\newcommand{\eps}{\varepsilon}

\renewcommand{\le}{\leqslant}
\renewcommand{\ge}{\geqslant}

\date{}

\author[Kolesnik]{Brett Kolesnik}
\address{University of Warwick, 
Department of Statistics}
\email{brett.kolesnik@warwick.ac.uk}

\author[Makai]{Tam\'as Makai}
\address{LMU Munich, Department of Mathematics}
\email{makai@math.lmu.de}

\author[Nenadov]{Rajko Nenadov}
\address{University of Canterbury,
School of Mathematics and Statistics}
\email{rajko.nenadov@canterbury.ac.nz}

\author[P\'erez-Gim\'enez]{Xavier P\'erez-Gim\'enez}
\address{University of Nebraska--Lincoln, 
Department of Mathematics}
\email{xperez@unl.edu}

\author[Pra{\l}at]{Pawe{\l} Pra{\l}at}
\address{Toronto Metropolitan University,
Department of Mathematics}
\email{pralat@torontomu.ca}

\author[Zhukovskii]{Maksim Zhukovskii}
\address{University of Sheffield,
School of Computer Science}
\email{m.zhukovskii@sheffield.ac.uk}


\begin{document}

\title[Critical activation density
in graph bootstrap percolation]
{The critical activation density
in graph bootstrap percolation}

\dedicatory{Dedicated to Brendan McKay and Nick Wormald}

\begin{abstract} 
In graph bootstrap percolation, edges of an Erd{\H o}s--R\'enyi 
random graph $\Gnp$ are initially active, and activation 
spreads to other edges of $K_n$ via the combinatorics of a fixed 
graph $H$: an edge becomes active whenever it is the unique 
inactive edge in a copy of $H$. The process $H$-percolates if all 
edges of $K_n$ are eventually activated. While classical cases 
such as $H=K_3$ (connectivity) and $H=K_4$ (related to $2$-neighbor 
bootstrap percolation) have been studied extensively, general graphs 
$H$ can exhibit wildly different behaviors.

In this work, we determine the critical $H$-percolation threshold 
$p_c(n,H)$ for every graph $H$, fully resolving a longstanding 
open question of Balogh, Bollob\'as, and Morris. The location of 
$p_c(n,H)$ is governed by a new, universal parameter $\rho(H)$, 
which measures the maximal efficiency of witness graphs that activate 
an edge.

To achieve this, we introduce a novel framework based on the 
unfolding and refolding of witness graphs. While previous works 
were restricted to specific families of $H$, our approach provides 
a unified strategy for all $H$. Inspired by algebraic topology, we 
lift witness graphs to covering graphs and algorithmically embed 
folded versions into $\Gnp$ via a sequence of extensions. 
Crucially, this allows us to incorporate highly efficient witness 
graphs of unbounded size, which are potentially far larger than 
$\Gnp$ itself.

Beyond resolving $p_c(n,H)$, our framework recovers and strengthens 
several existing bounds in the literature. Finally, we initiate the 
study of the universal density parameter $\rho(H)$ and pose central 
open questions regarding its computability and its exact correspondence 
with the sharpness of the $H$-percolation threshold.
\end{abstract}

\maketitle

\section{Introduction}
\label{S_Intro}

In highly influential works, 
Ulam \cite{U50} and von Neumann \cite{vN66} 
introduced cellular automata in an attempt 
to understand the mystery of biological self-replication. 
Although such processes evolve according to local rules, 
they can nonetheless lead to complex global behavior.
Perhaps the most famous example is Conway's {\it Game of Life} \cite{Gar70}, 
in which, at any given moment, each 
square in an infinite square grid
takes value $0$ or $1$ 
(representing the absence and presence of life). 
The game consists of a simple set of rules by which squares change values, 
depending only on the values of those nearby.  
Despite its apparent simplicity, the game gives rise to a wide diversity of behavior, 
and it is generally undecidable whether a game will ever end.

Shortly after the work of von Neumann, 
in the context of extremal combinatorics, 
Bollob\'as \cite{Bol68} introduced an intriguing 
class of dynamics that evolve in 
graphs. The update rule for this model is 
governed by a fixed graph $H$. 
To begin, we start with a graph $G$ and let $K_G$ 
denote the clique on $V(G)$, the vertex set of $G$. 
The edges in $G_0=G$ are said to be {\it initially active}. 
In each step $t\ge1$,
a graph $G_t$ induced by all edges activated by time $t$ is obtained as follows: 
At time $t\ge1$, we {\it activate} each edge $e$ such that 
$H_e\setminus e\subseteq G_{t-1}$
for some copy 
$H_e\subseteq K_G$ of $H$ containing $e$, 
and obtain $G_t$ by adding all such edges $e$ to $G_{t-1}$. 
In other words, an edge becomes active if it is the only inactive edge in an otherwise
active copy of $H$ in $K_G$. 
The resulting graph $\cup_{t\ge0} G_t$ with all {\it eventually active} edges, 
obtained once the 
{\it $H$-dynamics} stabilize, 
is denoted by $\langle G\rangle_H$. 
The graph $G$ is said to 
be {\it weakly $H$-saturated} if $\langle G\rangle_H=K_G$. 
The main result in \cite{Bol68} identifies, for all $3\le r\le 6$,  the 
minimum  number 
${\rm wsat}(n,K_r) 
=(r-2)n-{r-1\choose2}$ 
of edges in a weakly $K_r$-saturated graph on 
$n$ vertices. This was later generalized, using a variety of methods, 
for all $r\ge3$, in independent works by Alon\cite{Alo85}, 
Frankl \cite{Fra82}, Kalai \cite{Kal84,Kal85}, and Lov\'{a}sz \cite{Lov77}. 

Remarkably, ${\rm wsat}(n,K_r)$ coincides with the minimal 
number ${\rm sat}(n,K_r)$ of edges
in a {\it strongly $K_r$-saturated} graph, 
with the property that 
adding {\it any} edge creates a copy of $K_r$ 
(i.e., one step of the $H$-dynamics turns $G$ into $K_G$). 
In identifying ${\rm sat}(n,K_r)$, 
Erd{\H o}s, Hajnal and Moon \cite{EHM64} 
showed there is a unique such graph, obtained 
by connecting each of $n-(r-2)$ many vertices to all vertices in a clique 
of size $r-2$. 
On the other hand, due to the more challenging, 
algorithmic nature of 
weak saturation, less is known about weakly $K_r$-saturated graphs. 
For instance, there is no explicit description of all weakly $K_r$-saturated graphs 
with ${\rm wsat}(n,K_r)$ many edges in the literature. 
One way to create such a graph, which we call 
an {\it $r$-Bollob\'as graph} (see \cite[Fig.~4]{Bol68}), 
is to start with a clique on $r-2$ vertices and then iteratively add the remaining 
$n-(r-2)$ vertices one at time, in each step adding edges to exactly 
$r-2$ previously
added vertices; but these are by no means
the only $K_r$-percolating graphs with the minimal number of edges.

One of the most well-studied of all cellular automata 
is called {\it bootstrap percolation}. This model was introduced 
by Chalupa, Leath and Reich \cite{CLR79}
to explain the sharp drop-off in the strength of a magnet observed 
as the concentration of a non-magnetic impurity reaches a critical value. 
A large literature has developed around the subject, which 
continues to remain active; see, e.g., the survey by Morris \cite{Mor17}. 
In bootstrap percolation, all vertices in a graph 
are initially {\it infected} independently with probability $p$.  
Further vertices become infected according to local rules.  
The most popular rule (although there are many variations) 
is the {\it $r$-neighbor rule}, under which a vertex becomes infected
once it has at least $r$ infected neighbors. 
Notice the connection with weak $K_r$-saturation: If the clique of 
size $r-2$ in an $r$-Bollob\'as graph is initially infected,  
then by induction all other vertices 
will become infected by the $(r-2)$-neighbor rule. 

More recently, Balogh, 
Bollob\'as and Morris \cite{BBM12} introduced 
{\it graph bootstrap percolation}, 
starting the $H$-dynamics with 
an Erd{\H o}s–Rényi random graph
$G=\Gnp$. 
In this context, if  $\GnpH=K_n$
we say that the process {\it $H$-percolates}.   
Recalling that sites are 
initially infected independently with probability 
$p$ in classical bootstrap percolation, 
it is thus quite natural to
initialize the $H$-dynamics with $\Gnp$, 
as then edges in $K_n$ are initially active independently with probability $p$. 
The  {\it critical $H$-percolation threshold}  is defined as  
\begin{equation}\label{E_pc}
p_c(n,H)=\inf\{p>0:\P(\GnpH=K_n)\ge1/2\}.
\end{equation}

As further motivation, we recall that weakly $K_r$-saturated graphs, 
and weakly $H$-saturated graphs more generally, 
are not well understood. As such, it is of interest
to study the structure of weakly $H$-saturated {\it random} graphs.  
When $p$ is close to $p_c$ the graph $\Gnp$
will, in some sense, favor the least-cost instances of such graphs. 
For example, when $\Gnp$ is a 
barely super-critical $K_4$-percolating graph it is likely to contain 
a $4$-Bollob\'as graph \cite{Kol22}. 
On the other hand, 
when $r\ge5$, barely super-critical $K_r$-percolating random graphs $\Gnp$
with high probability do not 
contain an $r$-Bollob\'as graph, 
but instead 
$K_r$-percolate in other ways that remain quite mysterious \cite{BKKP25}.

\subsection{Main results}
One of the main problems in this area, 
stated as Problem~1 in \cite{BBM12}, 
it to find the limit (assuming it exists) 
\[
\ell(H)=\lim_{n\to\infty}\frac{\log p_c(n,H)}{\log n},
\]
for {\it every} graph $H$. 
Our first result answers this question, which the
authors state \cite[p.\ 438]{BBM12} ``would represent a major step''
towards the ``ultimate aim of this line of research.'' 

\begin{definition}
For a graph $H$, let 
${\mathcal A}(H)$ be the set of all 
pairs $(e,A)$ 
such that $A$ is a graph and $e$ is an edge in $\langle A\rangle_H\setminus A$. 
We define the {\it critical $H$-activation density}  as 
\begin{equation}\label{E_rho}
\rho(H)=\inf_{(e,\,A)\in{\mathcal A}(H)}\, 
\max_{e\subset F\subseteq A\cup e}\, 
\frac{e(F)-1}{v(F)-2}.
\end{equation}
\end{definition}
Above, and throughout this work, we let $v(F)$ and $e(F)$ 
denote the number of vertices and edges in a graph $F$.

Note that, if $e(H)=0$ then $\rho(H)=\infty$ 
(by convention, since $\mathcal A(H)=\emptyset$); 
whereas, if $e(H)=1$ then $\rho(H)=0$. 
On the other hand, for $e(H)\ge 2$, 
one can easily verify that $0<\rho(H)<\infty$.
Throughout this work, we will for convenience 
assume that $H$ has no isolated
vertices, since adding isolated vertices to $H$
has no effect on $p_c(n,H)$, for all large enough $n$.

We note that the ``$-1$'' and ``$-2$'' in \eqref{E_rho} 
compensate for the edge $e$ and its vertices. 
Indeed, the quantity inside the infimum can be viewed
as the $2$-density of $G=A\cup e$, ``rooted'' at $e$. 
We recall that the {\it $2$-density} $m_2(G)$ of a graph $G$
is the maximum of $(e(F)-1)/(v(F)-2)$, over all 
subgraphs $F\subseteq G$ with $v(F)\ge3$. 
When maximizing in \eqref{E_rho}, we  require that $e\subset F$.

Our first main result locates $\ell(H)$. 

\begin{thm}\label{T_main}
For any graph $H$, 
\[
\ell(H)=-1/\rho(H).
\] 
\end{thm}

This result includes the trivial cases that $e(H)=0$ (when $\rho(H)=\infty$ and  
$\ell(H)=0$) and also $e(H)=1$ (when $\rho(H)=0$ and $\ell(H)= -\infty$),  
but we are mainly interested in the cases that $e(H)\ge2$ so that $0<\rho(H)<\infty$.
The cases that $\rho(H)\le 1$ are relatively simple. In Section \ref{S_HT}
below, we prove hitting time results for all such $H$. 
These results, in particular, 
improve upon some of the results for specific $H$ presented in 
\cite[\S5]{BBM12}. 
For the case that $\rho(H)>1$ we prove the following.

\begin{thm}\label{T_main2}
Suppose that $\rho=\rho(H)>1$. Then 
$p_c=p_c(n,H)$ satisfies 
\begin{equation}\label{E_P1}
\Omega(1/\log n)
=n^{1/\rho} p_c
=O(\log^2 n).
\end{equation}
\end{thm}

In fact, in Section \ref{S_UB} below, we  
prove a more technical upper bound  of 
\begin{equation}\label{E_ubpc}
n^{1/\rho} p_c = O(\log^{2/\rho+2/a} n), 
\end{equation}
where if $\rho$ is rational then $a/b$ is its
expression as an irreducible fraction, and if $\rho$ is 
irrational then $a$ can be taken to be arbitrarily large. 
For all $\rho>1$, we have
$a+\rho\le a\rho$, so \eqref{E_ubpc}
implies the upper bound in \eqref{E_P1}. 

We note that \eqref{E_ubpc} is slightly weaker than the 
bound $n^{1/\rho}p_c=O(\log^{4/r}n)$
proved in 
\cite[p.\ 416]{BBM12}
for cliques $H=K_r$, with $r\ge4$.
However, \eqref{E_ubpc} holds for {\it all}
graphs $H$, with $\rho>1$. 
In any case, we suspect (see Section \ref{S_Qs} below) that  
$p_c=O(n^{-1/\rho})$
might hold for all such $H$.

\subsection{Outline}
A brief overview of our proof strategy is given in Section~\ref{S_approach}, 
after recalling the notion of a witness graph in Section~\ref{S_WGs}.
In Section~\ref{S_balH}, 
we recall the notion of a balanced graph $H$, and discuss how our 
results apply in this well-studied case. 
A number of open problems are presented in Section~\ref{S_Qs}. 
In Section~\ref{S_LB}
we prove the lower bound
and in Section~\ref{S_UB}
we prove the upper bound in our main result, Theorem~\ref{T_main2}.
Finally, in Section~\ref{S_HT}, we prove hitting time results for all $H$ 
with a leaf or with $\rho(H)\le1$.

\subsection{Witness graphs}
\label{S_WGs}

We note that $\rho(H)$ is related to the notion of a 
{\it witness graph}, 
as introduced in 
\cite{BBM12}. 
For any graph $G$, 
and each edge $e$ in $\langle G\rangle_H$, 
the {\it witness graph algorithm (WGA)} \cite{BBM12}
finds an inclusion-minimal 
subgraph $W_e\subseteq G$
such that 
$e$ is in $\langle W_e\rangle_H$.
The WGA is an inductive procedure: 
\begin{itemize}[nosep]
\item Put $W_e=e$ for each initially active edge $e$ in $G_0=G$.
\item Then, for each edge $e$ in $G_t$, activated at some time $t\ge1$, 
we let $H_e$ (chosen arbitrarily if not unique) 
be the copy of $H$ that $e$ completes. We put
\[
W_e=\bigcup_{f\in E(H_e\setminus e)} W_f,
\]
noting that $H_e\setminus e\subseteq G_{t-1}$, 
and so $W_f$ has been defined by time $t-1$ for all edges $f$
in $H_e\setminus e$. 
\end{itemize}
By induction, for all edges $e$ in $\langle G\rangle_H\setminus G$, 
the witness graph $W_e$
is an inclusion-minimal 
subgraph $W_e\subset G$ for which 
$(e,W_e)\in{\mathcal A}(H)$. 
Therefore, if $(e,A)\in {\mathcal A}(H)$, then, by taking $G=A$, the WGA 
can be used to find an inclusion-minimal subgraph 
$W_e\subseteq A$ for which $(e,W_e)\in {\mathcal A}(H)$. 

Since \cite{BBM12}, 
most bounds on $p_c$ 
have proceeded as follows: 
\begin{itemize}[nosep]
\item 
To obtain an upper bound on $p_c$, a reasonably efficient   
class of witness graphs (typically of logarithmic size) is identified. 
If the correlations of such structures in $\Gnp$ can be controlled, they can be used to 
show that $H$-percolation occurs for all large enough $p$. 
\item To obtain a 
lower bound on $p_c$, a special property of the WGA is used, 
which is reminiscent of a property discovered by Aizenmann and Lebowitz
\cite{AL88} 
for $2$-neighbor bootstrap percolation on $\Z^2$: 
Specifically, if we define the {\it size} of a witness graph $W_e$ to be 
$v(W_e)-2$ (i.e., its number of vertices outside of $e$), then it can be seen
(see \cite[Lemma 13]{BBM12} and \cite[Lemma 8]{BK24}) 
that, started from any graph $G$, in each time step of the $H$-dynamics 
the maximum size of any witness graph defined so far 
increases by at most a factor of $e(H)$. Therefore, to show that 
some $e\in E(K_n)$ is, 
with high probability, 
{\it not} added by the $H$-dynamics
to $\Gnp$ it suffices to show that, with high probability, 
(1) $e$ has no witness graph of size at most 
$\gamma \log n$ (for an appropriately chosen constant $\gamma>0$)
and (2) no other edge $f$ has a witness graph of size between 
$\gamma \log n$ and $e(H)\gamma \log n$. 
\end{itemize}

\subsection{Our approach}
\label{S_approach}

A crucial barrier in applying the above proof strategy for arbitrary 
$H$ is that in optimally efficient  
$(e,A)\in{\mathcal A}(H)$, the maximum in~\eqref{E_rho} is 
not necessarily achieved for $F=A\cup e$. 
In particular, this influences the correlations mentioned above. 

When working with \eqref{E_rho}, we need to 
consider the {\it densest} subgraphs 
of witness graphs, rather than the witness graphs themselves. 
Recall that in the definition of $\rho(H)$, for each $A\in{\mathcal A}(H)$, 
we maximize the ratio 
$(e(F)-1)/(v(F)-2)$ over all subgraphs $e\subset F\subseteq A\cup e$. 
If some $F$ maximizes this ratio, we call $D=F\setminus e$ a {\it densest
part} of $A$ (with respect to $e$). 
For most graphs $H$ that have been studied, such as cliques $K_r$
(and, more generally, the balanced graphs $H$ discussed in Section \ref{S_balH}
below), in the most efficient witness graphs $W$
(e.g., in $K_r$-ladders, as in Figure \ref{F_ladder} below), the entire 
graph $D=W$ is {\it the} densest part of $W$. 

In proving the lower bound in Theorem~\ref{T_main2}, 
we show that 
an Aizenmann--Lebowitz-type property
continues to hold when looking at the dense parts $D$
of a general witness graph $W$ (for which possibly $D\neq W$). 
The key idea is to look at the 
{\it maximal} such $D_*\subseteq W$, and then argue that after each time step of WGA 
the maximum size over all such $D_*$ defined so far increases by at most 
a factor of $e(H)$. 

The upper bound in Theorem \ref{T_main2} is more complicated. 
We start by selecting some $(e,A_*)\in{\mathcal A}(H)$ such that 
\[
\max_{e\subset F\subseteq A_*\cup e}\, \frac{e(F)-1}{v(F)-2}
<\rho(H)+\frac{1}{\log n}.
\]
Note that $A_*$ depends on $n$, and is nearly as efficient as possible 
at activating a given edge $e$. The difficultly, however, lies in the fact
that we have no control whatsoever 
over the size of $A_*$. Indeed, $A_*$ may even be {\it much}
larger than $n$, the size of ${\mathcal G}_{n,p}$. As such, proof techniques
used in previous works do not apply. 
To overcome this challenge, 
we introduce a novel technique of embedding witness graphs in ${\mathcal G}_{n,p}$:
we will in some sense ``unfold'' the graph $A_*$ (to essentially obtain a covering graph
of $A_*$), and then recursively, one small
extension at a time, ``fold'' a modified version of $A_*$ into ${\mathcal G}_{n,p}$.

\subsection{Balanced graphs}
\label{S_balH}
Prior to the current work, the most general results
concerned the following class of graphs:  
Following \cite{BBM12}, we say that a graph $H$ is  
{\it balanced} if $v(H)\ge4$
and (cf.\ \eqref{E_rho} above) 
\begin{equation}\label{E_lam}
\frac{e(F)-1}{v(F)-2}\le \lambda(H):=\frac{e(H)-2}{v(H)-2}, 
\end{equation}
for all subgraphs $F\subset H$ with $3\le v(F)<v(H)$.\footnote{In \cite{BBM12} it is further 
assumed that $e(H)/2\ge v(H)-1$. This implies  $\lambda(H)\ge2$, which is used in the 
proof of Lemma 6 in \cite[p.\ 419]{BBM12}. However, in light of Lemma 16 in \cite{BK24}, 
only the minimal assumption that $v(H)\ge4$ is required.}

For instance, all cliques $K_r$ with $r\ge4$ are balanced. 
Naturally, $H$ is {\it strictly balanced} if \eqref{E_lam}
holds with $\le$ replaced by $<$. Cliques $K_r$ are
strictly balanced for all $r\ge5$; however, $K_4$ is balanced 
but not strictly. 
In some sense, ``most'' graphs $H$ are strictly balanced since,   
as shown by Bartha, Kronenberg and the 
first author \cite{BKK24},  
almost surely, 
 ${\mathcal G}_{\infty,1/2}[[r]]$ is  
strictly balanced 
for all large $r$, where ${\mathcal G}_{\infty,1/2}$ is the {\it Rado graph}, 
where every pair of vertices from the infinite set $\mathbb{Z}_{>0}$ 
is adjacent with probability $1/2$, 
and ${\mathcal G}_{\infty,1/2}[[r]]\sim \mathcal{G}_{r,1/2}$ 
is its induced subgraph on the first $r$ vertices. 

A main result in \cite[Theorem 1]{BBM12} shows that, 
for all $r\ge5$,  
\begin{equation}\label{E_Kr}
p_c(n,K_r)
= n^{-1/\lambda(K_r)+o(1)}. 
\end{equation}
The proof of the lower bound introduces the WGA, 
as described above, and also shows that  
when $H=K_r$, any witness graph $W$ with $v(W)$ vertices
has at least 
\begin{equation}\label{E_eWlb}
e(W)\ge\lambda(K_r)(v(W)-2)+1 
\end{equation}
many edges. 

On the other hand, 
the proof of the upper bound works for all balanced graphs
$H$, using a special type of witness graph called an {\it $H$-ladder}. 
An $H$-ladder $L$ of height $h$ is obtained by 
stringing together a series of $h$ many copies
of  $H$ in such a way that adjacent pairs intersect in a single edge. 
To obtain $L$, all such shared edges are removed, 
together with another edge $e$ from one of the copies of 
$H$ at an end of the series. 
By induction, the $H$-dynamics activate all such removed edges, 
and $L$ is a witness graph for $e$. 
An $H$-ladder $L$ of height $h$ has $v(L)=h(v(H)-2)+2$ 
many vertices and 
\[
e(L)
=h(e(H)-2)+1
=\lambda(H)(v(L)-2)+1
\] 
many edges; see Figure \ref{F_ladder} (cf.\ \cite[Fig.\ 1]{BBM12}). 

\begin{figure}[h]
\centering
\includegraphics[scale=1]{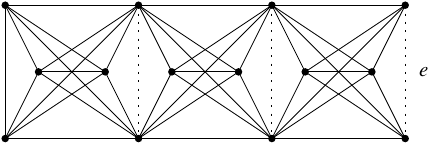}
\caption{A $K_6$-ladder $L$ of height $h=3$. 
All edges that join consecutive  
copies of $K_6$ 
and the edge $e$ at the end of the ladder 
are added to $L$ by the $K_6$-dynamics, 
so that $(e,L)\in{\mathcal A}(K_6)$.}
\label{F_ladder}
\end{figure}

In subsequent work by Bartha and the first author \cite{BK24}, 
it is shown that \eqref{E_Kr} extends to all balanced graphs $H$. 
A key step in this proof is to show (see \cite[Proposition 9]{BK24}) that, 
for any $H$ (with at least 4 vertices and minimum degree at least 2), 
a witness
graph $W$ with $v(W)$ many vertices has at least 
\begin{equation}\label{E_lamstar}
e(W)\ge\lambda_*(H)(v(W)-v(H))+e(H)-1
\end{equation}
many edges, where 
\[
\lambda_*(H)=\min_F\frac{e(H)-e(F)-1}{v(H)-v(F)}, 
\]
minimizing 
over all subgraphs $F\subset H$ with $2\le v(F)<v(H)$ 
vertices. 
In \cite[Theorem 4]{BK24}
it is shown that 
\[
p_c(n,H)\ge n^{-1/\lambda_*(H)+o(1)}. 
\]
Prior to the current work, this was the best general 
lower bound on $p_c$.

It can be shown (see \cite[Lemma 10]{BK24}) that 
$\lambda_*(H)\le \lambda(H)$, with equality if and only
if $H$ is balanced. 
Hence, in particular, \eqref{E_eWlb} above extends 
from $K_r$ to all balanced graphs $H$, noting that 
when $\lambda_*(H)=\lambda(H)$ the right hand side in 
\eqref{E_lamstar} simplifies to 
\[
\lambda(H)(v(W)-v(H))+e(H)-1
=\lambda(H)(v(W)-2)+1. 
\]
As such, as shown in \cite[Theorem 2]{BK24}, 
\begin{equation}\label{E_H}
p_c(n,H)
= n^{-1/\lambda(H)+o(1)},
\end{equation}
for all balanced graphs $H$. 

We note that, by Theorem~\ref{T_main} and the 
following observation, 
we recover the general result 
\eqref{E_H} for balanced graphs. 

\begin{prop}\label{P_rhoB}
For all balanced graphs $H$, 
we have that $\rho(H)=\lambda(H)$. 
\end{prop}

To prove this result, we will require the technical fact, mentioned above,  
that if $H$ is balanced then 
any witness graph $W$ has at least $\lambda(H)(v(W)-2)+1$ edges. 
Let us briefly sketch a proof (see \cite{BK24} for details) of this fact:
\begin{itemize}[nosep]
\item First, we slow down the $H$-dynamics so that in each step 
a single copy of $H$ is completed, and suppress 
the steps of WGA that do not contribute to $W$. This is called 
the {\it red edge algorithm (REA)} in \cite{BBM12}. 
\item Then, at any given time,  
the edges in $W$ (the black edges in REA) and those 
so far added to $W$ (the red edges in REA) by the $H$-dynamics
can be partitioned into a collection of edge-disjoint 
components, where the 
edges in two copies of $H$ are in the same component 
if the copies share an edge. 
\item By induction, it can be shown that 
each 
component $C$ has at least $\lambda(H)(v(C)-2)+1$ 
edges in $W$ (i.e., black edges). 
This proves the claim, since by the end of the REA 
there is only one component, the 
witness graph $W$. 
\item In the inductive step, suppose that a 
copy $H'$ of $H$ merges together with some number of 
components  $C_1,\ldots,C_m$ 
to form a new
component $C$. 
If $C_i$ has $v_i\ge 2$
vertices in common with $H'$ then, 
since $H$ is balanced, they have at most 
$\lambda(v_i-2)+1$ edges in common (see \eqref{E_lam} above), 
so the loss of vertices and edges caused by the 
overlap compensates for one another. 
\end{itemize}

\begin{proof}[Proof of Proposition \ref{P_rhoB}]
Suppose that $(e,A)\in{\mathcal A}(H)$. The WGA
gives an inclusion-minimal subgraph $W\subseteq A$ for 
which $(e,W)\in{\mathcal A}(H)$. 
Since $H$ is balanced,  
$e(W)\ge \lambda(H)(v(W)-2)+1$. 
Since $W$ contains the vertices in $e$ but not the edge $e$ itself,
the graph $F=W\cup e$ satisfies 
\begin{equation}\label{E_lamrho}
\frac{e(F)-1}{v(F)-2}
=\frac{e(W)}{v(W)-2}
\ge \lambda(H)+\frac{1}{v(W)-2}.
\end{equation}
Since there are, e.g., arbitrarily large
$H$-ladders for which \eqref{E_lamrho} is an equality, 
it follows, taking an infimum over all $(e,A)\in{\mathcal A}(H)$, 
that 
$\rho(H)=\lambda(H)$. 
\end{proof}

\subsection{Questions}\label{S_Qs}

The expression \eqref{E_rho} above for $\rho(H)$ 
is rather abstract. 
At this level of generality, a 
simpler expression may not exist. 
Currently, the most interesting 
example of a unbalanced 
graph $H$ for which the critical exponent has been  
explicitly identified
is the result 
that $p_c(n,K_{2,4})=\Theta(n^{-10/13})$
 by 
Bidgoli, Mohammadian, and Tayfeh-Rezaie
\cite{BMTR21}.

In recent work of 
Balister, Bollob\'as, Morris, and Smith
\cite{BBMS25}
it has been announced that forthcoming work 
will show that the exponent in $p_c$
for various monotone cellular automata
on Euclidean grids/lattices is uncomputable in general. 
This leads us to ask the following question. 

\begin{question}
Is  $\rho(H)$ computable
for general $H$? 
\end{question}

Although we do not know if $\rho(H)$ can be defined effectively, 
we also ask the following question. 

\begin{question}
Is $\rho(H)$ rational for all $H$?
\end{question}

A somewhat related question on possible values of the limit $\mathrm{wsat}(n,H)/n$ was asked 
by Terekhov and the sixth author \cite{TZ-combi}. In particular, it is known that this limit can be fractional 
(and therefore linear algebraic methods give suboptimal results; see \cite{TZ-matroids} by the same authors), 
but it is not known whether the limit is always rational. Some progress towards describing the 
possible values of the limit was recently achieved by Ascoli and He \cite{AH26}.

Bartha and the first author \cite{BK24}
have shown that 
$p_c(n,H)=O(n^{-1/\rho(H)})$
for all strictly balanced graphs 
(in which case, by Proposition \ref{P_rhoB} above, 
$\rho(H)=\lambda(H)$). 
It is known 
(see \cite[\S5]{BBM12} and Section \ref{S_leaf} below) 
that $p_c(n,H)=\Theta(n^{-1/\rho(H)})$
for all graphs $H$ with a {\it leaf} (i.e., a vertex of degree 1).  
We are not aware of any graphs $H$ for which 
$\rho(H)>1$ and 
$p_c(n,H)\gg n^{-1/\rho(H)}$. 

\begin{question}\label{Q_nolog}
Does the upper bound 
\[
p_c(n,H)=O(n^{-1/\rho(H)})
\]
hold for all graphs $H$ with $\rho(H)>1$?
\end{question}

On the other hand, as
shown in Section~\ref{S_rho1} below, 
if $\rho(H)=1$
and the minimum degree $\delta(H)=2$
then $p_c(n,H)\sim (\log n)/n$.

Finally, let us define 
\begin{equation}\label{E_pceps}
p_\eps(n,H)=\inf\{p>0:\P(\GnpH=K_n)\ge\eps\}.
\end{equation}
Observe (see \eqref{E_pc} above) that $p_c(n,H)=p_{1/2}(n,H)$. 
We say that the $H$-percolation threshold is {\it sharp}
if, for all small $\eps>0$,
\begin{equation}\label{E_sharp}
p_{1-\eps}(n,H)-p_\eps(n,H)\ll p_c(n,H),
\end{equation}
and call it {\it coarse} otherwise. 
Problem 2
in \cite{BBM12} asks when 
the $H$-percolation threshold is sharp in this sense. 
We believe that sharpness is 
related to whether the critical $H$-activation density is
ever attained for finite $A$. 

\begin{question}\label{C_sharp}
Is the $H$-percolation threshold sharp 
if and only if 
\begin{equation}\label{E_sharp}
\max_{e\subset F\subseteq A\cup e}\, \frac{e(F)-1}{v(F)-2}>\rho(H),
\end{equation}
for all $(e,A)\in{\mathcal A}(H)$?
\end{question}

The $K_r$-percolation threshold
is sharp for all cliques.
It is shown in \cite{BBM12} that 
$p_c(n,K_4)=\Theta((n\log n)^{-1/2})$
and, 
as noted in \cite{BBM12}, this result together with Friedgut \cite[Theorem 1.4]{Fri99}
implies sharpness. 
Further work \cite{AK18,AK21,Kol22} shows that 
$p_c(n,K_4)\sim(3n\log n)^{-1/2}$. 
Recent work by 
Bartha, the first author, Kronenberg, and Peled has shown, 
for all $r\ge5$, that  
$p_c(n,K_r)\sim\gamma_rn^{-1/\lambda(K_r)}$, where 
$\gamma_r$ is described in terms of the 
$({r\choose2}-1)$th
Fuss--Catalan numbers.
Note that, for $r\ge5$, there is no longer a polylogarithmic 
term in the asymptotics for $p_c(n,K_r)$, but nonetheless
the threshold continues to be  sharp. 
Since cliques are balanced it follows (see  
Proposition \ref{P_rhoB} and 
\eqref{E_lamrho} above) that \eqref{E_sharp}
holds when $H=K_r$, in line with 
Question \ref{C_sharp}. 

On the other hand, the $H$-percolation threshold is 
coarse, e.g., for all $H$ with a leaf; 
see Section \ref{S_leaf} below.

\subsection{Acknowledgements}

XPG was supported in part by NSF grant DMS-2201590.
PP was supported in part by NSERC DG grant RGPIN-2022-03804.
This project began at the MATRIX 
event 
{\it Combinatorics of McKay and Wormald}
in June 2025. 
We thank the organizers, Jane Gao, 
Catherine Greenhill, 
Mikhail Isaev, 
Anita Liebenau,  
Ian Wanless, 
and the MATRIX staff. 
This work is dedicated to the lasting influence of 
Brendan McKay and Nick Wormald
on the authors.

\section{Lower bound when $\rho(H)>1$}
\label{S_LB}

In this section, we prove the following result; cf.\ \cite[Proposition 8]{BBM12}
for the case $H=K_r$. 

\begin{thm}\label{T_mainLB}
Suppose that $\rho=\rho(H)>1$
and $n p^\rho \log^\rho n=\eps$.  
Then, for any sufficiently small $\eps>0$, we have 
that 
$\GnpH\neq K_n$ with high probability. 
\end{thm}

Throughout this section, we will assume that 
$\rho(H)>1$. In particular, this implies that 
$e(H)\ge v(H)$.

For graphs $A\subset B$, let 
\[
\rho(A,B)=\frac{e(B)-e(A)}{v(B)-v(A)}
\]
and put 
\[
\rho^{\rm max}(A,B)
=\max_{A\subset B'\subseteq B}
\rho(A,B'). 
\]
Note that 
\begin{equation}\label{e_rhomax}
\rho(H)
=\inf_{(e,\,A)\in{\mathcal A}(H)}
\rho^{\rm max}(e,A\cup e).
\end{equation}

The following observations are straightforward. We omit the proof.

\begin{lemma}\label{L_rho_comb}
Let $A\subset B\subset C$. Then, for any $\alpha>0$, 
\begin{enumerate}[nosep]
\item if $\rho(A,B)\ge\alpha$ and $\rho(B,C)\ge\alpha$ then $\rho(A,C)\ge\alpha$;
\item if $\rho(A,B)<\alpha$ and $\rho(B,C)<\alpha$ then $\rho(A,C)<\alpha$.
\end{enumerate}
\end{lemma}

\begin{lemma}\label{L_Fs}
Let $A$ be a graph. 
For each edge $e$ in $A$, put $W_e=F_e=e$. 
Recall that, for all other edges $e$ in 
$\langle A\rangle_H\setminus A$, the WGA 
identifies a copy $H_e$ of $H$ completed by $e$ and  
an inclusion-minimal 
subgraph 
\[
W_e=\bigcup_{f\in E(H_e\setminus e)}W_f\subset A
\]
for which $(e,W_e)\in{\mathcal A}(H)$. 
Then, for each such $e$, 
there is some  $e\subset F_e\subseteq W_e\cup e$, for which 
$\rho(e,F_e)\ge\rho(H)$ and 
\begin{equation}\label{E_Fexp}
v(F_e)-2\le\sum_{f\in E(H_e\setminus e)} v(F_f). 
\end{equation}
\end{lemma}

The ``$-2$'' in \eqref{E_Fexp} 
compensates for the vertices in $e$, 
as they may not be in any of the edges $f\in E(H_e\setminus e)$. 

\begin{proof}
We will prove this claim by induction, 
on the time at which an edge is activated  
by the $H$-dynamics starting with $A$.
In doing so, it will be helpful to also establish, 
for each edge $e$ in $\langle A\rangle_H\setminus A$, 
the existence of some $B_e$ such that either (1) $F_e=B_e$ or else 
(2) $F_e \subset B_e$ and
$\rho^{\rm max}(F_e,B_e)<\rho(H)$, 
and in either case $(e,B_e\setminus e)\in{\mathcal A}(H)$. 
Let us stress  that $B_e\setminus e$ is not necessarily a subgraph of $A$,
but is rather a technical device that will be used in the proof. 

For an edge $e$ added at time $t=1$, 
we simply take $B_e=H_e$
and let $F_e$ be an inclusion-maximal subgraph 
$e\subset F_e \subseteq H_e$
for which $\rho(e,F_e)\ge \rho(H)$. 
We note that \eqref{E_Fexp} holds, since $\rho(H)>1$
and so, as noted above, $e(H)\ge v(H)$. 
If $F_e\neq B_e$ then, by  
Lemma \ref{L_rho_comb}(1), 
it follows 
by the choice of $F_e$ 
that $\rho^{\max}(F_e,B_e)<\rho(H)$. 

Next, if some $e$ is added at time $t>1$, 
then all $f\in E(H_e\setminus e)$
were added by time $t-1$, 
so the induction hypothesis applies to these edges. 
To begin, we put 
\[
F_e'=e\cup \bigcup_{f\in E(H_e\setminus e)} F_f\setminus f, 
\]
and obtain $B_e$ from $F_e'$ by attaching a disjoint copy 
of $B_f\setminus F_f$ to each $F_f\setminus f$ above. 
By construction, $(e,B_e\setminus e)\in{\mathcal A}(H)$. 

First, note that if all the $B_f=F_f$ then $B_e=F_e'$. 
In this case, since $(e,B_e\setminus e)\in{\mathcal A}(H)$, 
we can take $F_e$ to be an 
inclusion-maximal subgraph
$e\subset F_e\subseteq F_e'$ satisfying $\rho(e,F_e)\ge\rho(H)$. 

On the other hand, suppose that at least one of the $B_f\neq F_f$. 
Then, by construction, $\rho^{\rm max}(F_e',B_e)<\rho(H)$. 
We claim that $\rho^{\rm max}(e,F_e')\ge\rho(H)$. 
Indeed, otherwise it would 
follow 
by Lemma \ref{L_rho_comb}(2)
that $\rho^{\rm max}(e,B_e)<\rho(H)$, in contradiction 
to $(e,B_e\setminus e)\in{\mathcal A}(H)$. 
To conclude, we let $F_e$ be an inclusion-maximal subgraph
$e\subset F_e\subseteq F_e'$ satisfying $\rho(e,F_e)\ge\rho(H)$.
Once again, by Lemma \ref{L_rho_comb}(1), 
it follows that $\rho^{\max}(F_e,B_e)<\rho(H)$. 
\end{proof}

The previous lemma has the following 
useful consequence. 

\begin{corollary}
Let $A$ be a graph. 
Consider the graphs $F_e$ in 
Lemma \ref{L_Fs}. Suppose that, for some $e$ in $\langle A\rangle_H\setminus A$, 
we have $v(F_e)>v(H)e(H)$. 
Then, for some $f$ that is added to $A$ before $e$ during the WGA, 
we have 
\[
v(F_e)/e(H)
\le v(F_f)<v(F_e).
\]
\end{corollary}

\begin{proof}
This follows noting, by \eqref{E_Fexp} in Lemma \ref{L_Fs}, 
that in each time step of the WGA
the maximum size over all $F_e$ defined so far increases by 
at most a factor of $e(H)$. 
Therefore, if the WGA eventually defines an $F_e$ of size $v(F_e)$, 
then there must have previously 
been defined some smaller $F_f$ of size at least $v(F_e)/e(H)$. 
\end{proof}

We are ready to prove the main result of this section. 
This proof is inspired by the proof of Proposition 8 in \cite{BBM12}. 
The crucial difference, however, is that we look at the 
dense parts $F$ of witness graphs $W$, as given by 
Lemma \ref{L_Fs} above, rather than the witness graphs
$W$ themselves. 

\begin{proof}[Proof of Theorem \ref{T_mainLB}]
Suppose that $np^\rho \log^\rho n=\eps$, 
for a sufficiently small constant $\eps>0$, 
as determined below. 

We apply the WGA to $\Gnp$
and use the graphs $F_e$ in Lemma \ref{L_Fs}.
Suppose that some given $e$ is in $\GnpH$.  Then either:
\begin{enumerate}[nosep]
\item $v(F_e)\le e(H)\log n$; or else  
\item some edge $f$  satisfies $\log n\le v(F_f)\le e(H)\log n$. 
\end{enumerate} 

In the first case, there would be some graph $F$ containing the 
vertices of $e$ and $k\le e(H)\log n$ many other vertices
and at least $\rho k+1$ edges. Taking a union bound, 
the probability of this event is at most 
\begin{align*}
\sum_{k=0}^{e(H)\log n} n^k {(k+2)^2\choose \rho k+1}p^{\rho k+1}
&\le O(p \log n) \sum_{k=0}^{e(H)\log n}(Cnp^\rho \log^\rho n)^k\\
&= O(p \log n) \sum_{k=0}^{\infty}(C\eps)^k, 
\end{align*}
for some constant $C$ depending only on $H$.
For any small enough $\eps>0$, the above expression 
is $O(p\log n)\to0$
as $n\to\infty$. 

Likewise, in the second case, the probability is bounded by 
\[
O(pn^2\log n) \sum_{k=\log n}^{e(H)\log n}(C\eps)^k
= O(n^{2-1/\rho+\log (C\eps)}\log^2 n)\to0, 
\]
for any small enough $\eps>0$. 
\end{proof}

\section{Upper bound when $\rho(H)>1$}
\label{S_UB}

In this section, we will prove the following result. 

\begin{thm}\label{T_UB}
Fix a graph $H$, with $\rho=\rho(H)>1$. 
If $\rho$ is rational, let $a/b$ be its expression as an irreducible fraction. 
If $\rho$ is irrational, let $b$ be arbitrarily large.  
Suppose that $np^\rho=A\log^{2+2/b} n$.
Then, for large enough $A$, with high probability, we have that 
$\langle\Gnp \rangle_H=K_n$.
\end{thm}

Throughout this section,  
we fix some $(e^*,A^*)\in{\mathcal A}(H)$ for which  
\begin{equation}\label{E_alphastar}
\rho^{\rm max}(e^*,A^*\cup e^*)<\rho+1/\log n.
\end{equation}
Such an $A^*$ exists by the definition \eqref{e_rhomax} of $\rho$. 
Furthermore, we let $W^*\subseteq A^*$ denote the witness graph for $e^*$
found by the WGA, as discussed in Section~\ref{S_WGs} above.
By \eqref{E_alphastar}, $W^*$ is an efficient witness graph. 

In proving Theorem \ref{T_UB}, the overall plan is to use 
$W^*$ as a basis to show that all missing edges in $\Gnp$ are 
activated eventually. 
For instance, if $W^*$ was small, with only $v(W^*)=O(\log n)$ many vertices,  say, 
we might naturally try to embed a copy of $W^*$ into $\Gnp$ ``based'' at each 
missing edge $e$
in $\Gnp$ (i.e., with the vertices of $e^*$ mapped to those of $e$). 
However, in appealing to \eqref{e_rhomax}, 
we have no control on $v(W^*)$. 
In particular, $v(W^*)$ could even be much larger than $n$, 
the size of $\Gnp$ itself. 

The main idea behind our proof of  Theorem \ref{T_UB}
is to use suitably ``compressed'' versions of 
$W^*$ to activate all missing edges in $\Gnp$. 
Roughly speaking, this will be achieved
by ``unfolding'' $W^*$ and then
``folding'' it into $\Gnp$ by recursively embedding a series of small extensions. 
Crucially, the ``unfolded'' version of $W^*$ preserves 
the local structure of $W^*$, so that
activation patterns at all stages of the process are maintained. 
Indeed, the ``unfolded'' version of $W^*$ can be viewed as its covering graph, and we will show that there exists a covering map from this ``unfolded'' version 
of $W^*$ into $\Gnp$ that is based at each missing edge of $\Gnp$.

\subsection{Extensions and cores}
The first step towards our 
proof of Theorem \ref{T_UB}
is to formally introduce the 
notion of an extension.

\begin{definition}
Let $S\subset F$ be graphs and suppose that 
$V(F)$ is ordered. 
Let $S'\subset F'\subseteq G$ be some other graphs, where $V(G)$ is ordered.  
Then we call $F'$ an {\it $(S,F)$-extension}
of $S'$ in $G$ if  there is an isomorphism of the graphs $F\setminus E(S)$
and $F'\setminus E(S')$ such that 
the $i$th largest vertex in $S$ is mapped to the $i$th 
largest vertex in $S'$, for each $1\le i\le v(S)$.
\end{definition}

In other words, $F'$ extends $S'$ in the same way as
$F$ extends $S$. 
Note that, in the definition above,  $F'$ is not necessarily an induced subgraph of $G$, 
and that the edges in $F[S]$ and $G[S']$ play 
no role.

\begin{definition}
For graphs $S\subset F$ and some $\alpha>0$, we say that $(S,F)$ 
is {\it $\alpha$-dense} if 
\[
e(F)-e(F')
\ge \alpha(v(F)-v(F')), 
\]
for {\it every} induced subgraph $S\subseteq F'\subset F$. 
\end{definition}

In other words, for any partial revealment  $F'\supseteq S$ of $F$, 
we will need to add at least $\alpha$ edges per vertex
in order to reveal the rest of $F$. In light of this, 
let us make the following observation. 

\begin{lemma}\label{L_mergeFs}
Let $e\subset F_1,F_2$ and suppose that $(e,F_1)$
and $(e,F_2)$ are $\alpha$-dense. 
Then $(e,F_1\cup F_2)$ is $\alpha$-dense. 
\end{lemma}

\begin{proof}
Indeed, let $F'$ be an induced subgraph of 
$F_1\cup F_2$ containing $e$. 
Let $x$ be the number of vertices in $F_1$ that are not in
$F'$, and let $y$ be the number of vertices in $F_2$ that are not in 
$F'\cup F_1$. Then 
there are at least $\alpha x$ many edges in $F_1$ with at most
one endpoint in $F'$, and at least $\alpha y$ many edges in $F_2$
with at most one endpoint in $F'\cup F_1$. Hence 
\[
e(F_1\cup F_2)-e(F')
\ge
\alpha (x+y)
=\alpha (v(F_1\cup F_2)-v(F')), 
\]
as required; see Figure \ref{F_dense}. 
\end{proof}

\begin{figure}[h]
\centering
\includegraphics[scale=0.85]{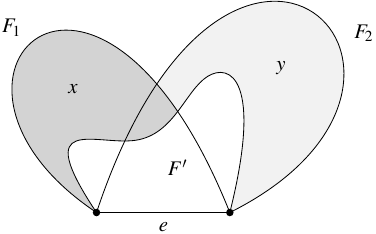}
\caption{Illustration of the proof of Lemma \ref{L_mergeFs}. 
If there are $x$ vertices (dark grey) in $F_1$ that are not in $F'$ 
and $y$ vertices (light grey) in $F_2$ that are not in $F'\cup F_1$, then there
are at least $\alpha x$ edges in $F_1$ with at most one endpoint in $F'$, 
and at least $\alpha y$ edges in $F_2$ with at most one endpoint in 
$F'\cup F_1$. 
}
\label{F_dense}
\end{figure}

\begin{definition}
\label{def:C-core}
Let $(e,A)\in{\mathcal A}(H)$. 
The {\it $\alpha$-core} $C_e\subseteq A$ of $(e,A)$
is defined as follows. 
If there is no subgraph $e\subset F\subseteq A\cup e$ such that 
$(e,F)$ is $\alpha$-dense,  
let $C_e$ be the empty graph on the vertices of $e$. 
Otherwise, let $C_e=F_e\setminus e$
where $F_e$ is the maximum such $F$. 
\end{definition}

Note that, 
by Lemma \ref{L_mergeFs}, $C_e$ is well defined.

\begin{definition}
Let $S\subset F$. Naturally, we say that  $(S,F)$ is 
{\it $\alpha$-sparse} if $\rho^{\rm max}(S,F)<\alpha$; that is, 
if 
\[
e(F')-e(S)<\alpha(v(F')-v(S)), 
\]
for {\it every} $S\subset F'\subseteq F$. 
\end{definition}

\begin{lemma}\label{L_sparse}
Let $(e,A)\in{\mathcal A}(H)$. Suppose that 
$C_e\subset A$ is the $\alpha$-core of $(e,A)$. 
Then $(C_e,A)$ is $\alpha$-sparse. 
\end{lemma}

\begin{proof}
For simplicity, put $C=C_e$. 
Indeed, towards a contradiction, let $F$ be a minimal 
$C\subset F\subseteq A$ for which 
\begin{equation}\label{E_Fdense}
e(F)-e(C)\ge \alpha(v(F)-v(C)).
\end{equation}
Then we claim that $(e,F)$ is $\alpha$-dense, 
in contradiction to the fact that $C$ is the
$\alpha$-core (the maximum such $F$). To see this, 
first note that, 
by the minimality of $F$, for any $C\subset F''\subset F$
we have that 
\begin{equation}\label{E_F''sparse}
e(F'')-e(C)<\alpha(v(F'')-v(C)).
\end{equation}
Next, consider some $e\subset F'\subset F\cup e$.
Suppose there are $x$ many vertices in $C$
that are not in $F'$, 
and $y$ many vertices in $F$ that are not in $F'\cup C$. 
Since $C$ is the $\alpha$-core, 
there are at least $\alpha x$ many edges
in $C$ with at most one endpoint in $F'$. 
By \eqref{E_Fdense} and \eqref{E_F''sparse}, 
there are at least $\alpha y$ many edges in $F$
with at most one endpoint in $F'\cup C$. 
Hence 
\[
e(F)-e(F')\ge\alpha(x+y)=\alpha(v(F)-v(F')), 
\] 
and so $(e,F)$ is $\alpha$-dense, 
yielding the desired contradiction;
see Figure \ref{F_sparse}. 
\end{proof}

\begin{figure}[h]
\centering
\includegraphics[scale=0.85]{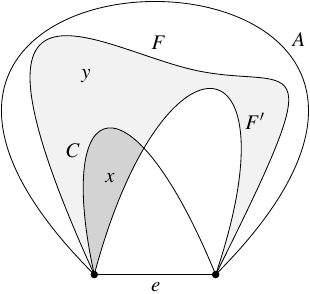}
\caption{Illustration of the proof of Lemma \ref{L_sparse}. 
If there are $x$ vertices (dark grey) in $C$ that are not in $F'$ 
and $y$ vertices (light grey) in $F$ that are not in $F'\cup C$, then there
are at least $\alpha x$ edges in $C$ with at most one endpoint in $F'$, 
and at least $\alpha y$ edges in $F$ with at most one endpoint in 
$F'\cup C$. 
}
\label{F_sparse}
\end{figure}

\subsection{Red and black edges}

As discussed in Section \ref{S_WGs}, 
the WGA applied to $A^*$ begins by setting $W_e^*=e$ 
for each 
initially active (at time $t=0$) edge 
$e$ in $A^*$. 
Then, for each edge $e$ activated at some time $t\ge1$
it selects (arbitrarily if not unique)
some copy $H_e^*$ completed by $e$ at time $t$
and sets 
\begin{equation}\label{E_We}
W_e^*=\bigcup_{f\in E(H_e^*\setminus e)}W_f^*.
\end{equation}
Since all $f$ in $H_e^*\setminus e$ have been activated by 
time $t-1$, all $W_f^*$ have already been defined, 
and so $W_e^*$ is well defined. 
Note that the witness graph $W^*\subseteq A^*$ 
discussed above
is $W^*=W^*_{e^*}$. 

The following definition 
is similar to the {\it red edge algorithm (REA)} in 
\cite{BBM12}, where we run the WGA but suppress
all witness graphs that do not contribute to the formation 
of the witness graph $W^*$ of interest. 

\begin{definition}
For each $e$ in $A^*$, put ${\mathcal R}_e^*=\emptyset$
and recall that $W_e^*=e$.  
For edges $e$ activated at time $t\ge1$, 
we put 
\[
{\mathcal R}_e^*=\{e\}\cup\bigcup_{f\in E(H_e^*\setminus e)}{\mathcal R}_f^*. 
\]
We call the edges in ${\mathcal R}_e^*$ the {\it red edges}
of $W_e^*$. 
The edges in $W_e^*$ itself are its {\it black edges}. 
Note that if the black edges of $W_e^*$ are
initially active then the $H$-dynamics will 
activate each red edge of $W_e^*$, towards
eventually activating $e$ itself. We let ${\mathcal R}^*={\mathcal R}^*_{e^*}$
denote the set of red edges of $W^*$. 
\end{definition}

\subsection{Small cores}
Throughout the remainder of this section, 
we let 
\[
\alpha_*=\rho+(\rho+2)/\log n.
\] 
Our next key observation is that, 
although we have no control over the  
size of $W^*$ itself, it is however the case that 
$W^*_f$ has a small $\alpha_*$-core 
for each red edge $f$ of $W^*$. 
As we will see, this will allow us to embed 
a ``folded'' version of $W^*$, one small extension at a time, 
based at each missing edge of $\Gnp$. 

\begin{definition}
For each red edge
$f\in{\mathcal R}^*$ of $W^*$,
we let $C_f^*$ denote the 
$\alpha_*$-core of $(f,W^*_f)$. 
If $f\in E(W^*)$ is a black edge of $W^*$, we
put $C_f^*=f$.
\end{definition}

\begin{lemma}\label{L_smallC}
For all large $n$, 
we have $v(C_f^*)\le 2\log n$, 
for all $f\in{\mathcal R}^*$. 
\end{lemma}

\begin{proof}
Suppose that, for some $f\in{\mathcal R}^*$, 
there is some subgraph
$f\subset F\subset W^*_f\cup f$ 
for which 
\[
\frac{e(F)-1}{v(F)-2}\ge \rho+\frac{\rho+2}{\log n}.
\]
Recall that $W_f^*\subset W^*$. 
Therefore, $e^*\subset F'\subset W^*\cup e^*$, 
where $F'=(F\setminus f)\cup e^*$. 
Also recall that $W^*\subseteq A^*$. 
Hence, by \eqref{E_alphastar}, and 
since $e(F')=e(F)$ and $v(F')\le v(F)+2$, 
it follows that 
\[
\frac{e(F)-1}{v(F)}\le \frac{e(F')-1}{v(F')-2}\le \rho+\frac{1}{\log n}.
\] 
Altogether, we find that 
\[
\frac{e(F)-1}{v(F)}-\frac{1}{\log n}
\le \rho
\le 
\frac{e(F)-1}{v(F)-2}-\frac{\rho+2}{\log n},
\]
and so 
\[
\frac{\rho+1}{\log n}
\le
\frac{2}{v(F)}\cdot\frac{e(F)-1}{v(F)-2}
\le \frac{2}{v(F)-2}
\left(\rho+\frac{1}{\log n}\right).
\]
Hence
\[
v(F)\le \frac{2\rho}{\rho+1}\log n+O(1)\le 2\log n,
\]
for all large $n$, and the result follows. 
\end{proof}

\subsection{Modified cores}
As discussed, we aim to embed a 
``folded'' version of $W^*$ based at each 
missing edge $e$ in $\Gnp$, and we plan to do so
one small extension at a time. In order to allow 
for such a recursive construction, it will 
be useful to consider the following
modified version of the $\alpha$-core. 

\begin{definition}
For each red edge $f\in {\mathcal R}^*$ in $W^*$, 
we let  
\begin{equation}\label{E_Ustar}
M_f^*:=\bigcup_{g\in E(H_f^*\setminus f)} C^*_g 
\end{equation}
denote the {\it modified $\alpha_*$-core} of $f$.
In particular, we let $M^*=M^*_{e^*}$ denote the 
{\it modified $\alpha_*$-core} of $e^*$.
\end{definition}

Note that, by Lemma \ref{L_smallC}, for all large $n$, 
\begin{equation}\label{E_hatC2}
v(M_f^*)\le \sum_{g\in E(H_f^*\setminus f)}v(C^*_g)
\le 2e(H)\log n,
\end{equation}
for all $f\in{\mathcal R}^*$.

\subsection{Extensions in $\Gnp$}
The following technical result will be used in 
our proof below of Theorem \ref{T_UB}.
Informally, it says that, 
with high probability, 
for {\it all} sufficiently small and sparse $(S,F)$-extensions, 
{\it all} subgraphs $S'\subset\Gnp$ have an $(S,F)$-extension
in $\Gnp$. This fact will allow us to recursively 
build a ``folded'' version of $W^*$ starting from each missing
edge in $\Gnp$. 

\begin{lemma}\label{L_Janson}
Let $\rho$ and $np^\rho=A\log^{2+2/b} n$ be as in Theorem \ref{T_UB}. 
Then, for large enough $A$, with high probability, for every $(S,F)$ 
such that 
\begin{itemize}[nosep]
\item $V(F)$ is ordered, 
\item $v(F)\le 2e(H)\log n$, and 
\item $(S,F)$ is $\alpha_*$-sparse, 
\end{itemize}
every subgraph $S'\subset\Gnp$
with $v(S')=v(S)$ has an $(S,F)$-extension  
in $\Gnp$ (where  $V(S')$
is ordered by the natural order of the integers). 
\end{lemma}

\begin{proof}
For technical convenience, we will restrict our attention to 
{\it ordered} $(S,F)$-extensions, such that,   
\begin{itemize}[nosep]
\item as above, the $i$th largest vertex in $S$ is mapped to the $i$th 
largest vertex in $S'$, for each $1\le i\le v(S)$, but that also 
\item the $j$th largest vertex in $V(F)\setminus V(S)$ is mapped
to the $j$th largest vertex in $V(F')\setminus V(S')$, for each $1\le j\le v(F)-v(S)$. 
\end{itemize}
As such, we essentially ignore the symmetries of $F$ outside of $S$, 
and consider one particular way of extending $S$. This simplifies
various calculations in our analysis. 

Fix some $(S,F)$, as above, and some 
$V\subset[n]$ of size $|V|=v(S)$. 
Let $S'$ denote the subgraph of $\Gnp$ induced by $V$, 
and let $X$ denote the number of 
ordered $(S,F)$-extensions of $S'$ in $\Gnp$. 

For ease of notation, put $B=2e(H)$ and note that 
$v(F)\le B\log n$. Also note that, since $(S,F)$ is $\alpha_*$-sparse, 
it follows that 
\[
\rho^{\rm max}(S,F)\le \rho+C\rho/\log n, 
\]
where $C=1+2/\rho$.

Applying Janson's inequality, we will show that, 
for all large $n$, 
\begin{equation}\label{E_Janson}
-\log \P(X=0)>B^2\log^2 n.
\end{equation}
Taking a union bound, this will prove the lemma, since 
the number of relevant $(S,F)$ and $V$ is at most
\[
2^{B\log n+(B\log n)^2/2}n^{B\log n}
\ll e^{B^2\log^2 n}. 
\]
Recall (see, e.g., \cite{AS16}) that Janson's inequality implies that 
\[
-\log \P(X=0)\ge \frac{\mu^2}{\mu+\Delta}
=\frac{1}{1/\mu+\Delta/\mu^2},
\]
where $\mu=\E X$ is the expected
number of ordered $(S,F)$-extensions of $S'$ in $\Gnp$, 
and $\Delta$ is the expected number of pairs of 
distinct ordered $(S,F)$-extensions of $S'$ in $\Gnp$ that share at least one edge
outside of $S'$. 
To prove \eqref{E_Janson}, we will show that 
$\mu\gg \log^2n$ and $\mu^2/\Delta> 2B^2\log^2n$. 

Let $s=v(S)$ and $t=v(F)-v(S)$. 

First, to see that 
$\mu\gg \log^2 n$, we note 
(using ${n\choose k}\ge(n/k)^k$) that 
\[
\mu= {n-s\choose t}p^{e(F)-e(S)}
\ge (1-o(1)) \left(\frac{np^{\rho^{\rm max}(S,F)}}{t}\right)^{t},
\]
since there are at most $t\rho^{\rm max}(S,F)$ many edges
in $F$ that are not in $S$. 
Next, we observe that 
\begin{align*}
np^{\rho^{\rm max}(S,F)}
&\ge
\left(\frac{A\log^{2+2/b} n}{n}\right)^{C/\log n}
A\log^{2+2/b} n\\
&=(1+o(1))e^{-C}A\log^{2+2/b} n.
\end{align*}
Since $(np^{\rho^{\rm max}(S,F)}/t)^t$ increases in $t$, 
we find that $\mu\gg\log^2 n$, as claimed. 

Finally, we turn to $\mu^2/\Delta$. 
Let $(F_i,i\ge1)$ be an enumeration of 
all possible ordered $(S,F)$-extensions 
of $S'$ in $\Gnp$. We let $T_i\subset F_i$ denote the subgraph induced
by the edges of $F$ with at least one endpoint outside of $V$. 
Let us write $i\sim j$ if $T_i,T_j$
share at least one edge. 
Then 
\[
\Delta=\sum_{i\sim j}\P(T_i,T_j\subset\Gnp)
=\mu\sum_{j\sim 1}\P(T_j\subset\Gnp|T_1\subset\Gnp). 
\]
Therefore, to show that 
$\mu^2/\Delta>2B^2\log^2n$, 
it suffices to 
verify  that 
\begin{equation}\label{E_corr}
\frac{1}{\mu}\sum_{j\sim 1}\P(T_j\subset\Gnp|T_1\subset\Gnp)
<\frac{1}{2B^2 \log^2n}.
\end{equation}
By considering $F_j$ with $x$ vertices in common with 
$F_1$, we claim that 
\begin{equation}\label{E_corrsumx}
\frac{1}{\mu}\sum_{j\sim 1}\P(T_j\subset\Gnp|T_1\subset\Gnp)
\le 
\sum_x\frac{{t\choose x}{n\choose t-x}}{{n-s\choose t}}
p^{-\lfloor x\rho+xC\rho/\log n\rfloor}.
\end{equation}
Indeed, this follows since, if $F_j,F_1$ have $x$ vertices in common, then they can 
share at most $\lfloor x\rho^{\max}(S,F)\rfloor$ many edges. 

Next, since the relevant $s$ and $t$ are of order $\log n$, 
we note 
(using $(n-k)^k/k!\le {n\choose k}\le n^k/k!$) that
\[
\frac{{t\choose x}{n\choose t-x}}{{n-s\choose t}}
\le \frac{n^t}{(n-s-t)^t} (t/n)^x{t\choose x}
\le (1+o(1)) \left(\frac{et^2}{xn}\right)^x.
\]
In particular, if $x\ge 2b+1$, then
\begin{align*}
\frac{{t\choose x}{n\choose t-x}}{{n-s\choose t}}
p^{-\lfloor x\rho+xC\rho/\log n\rfloor}
&\le \left(\frac{O(t^2)}{np^{\rho}}\right)^x
\le \left(\frac{O(1)}{\log^{2/b}n}\right)^x\\
&=O(\log^{-4}n)\left(\frac{O(1)}{\log^{2/b}n}\right)^{x-2b}.
\end{align*}
On the other hand, if $x\le 2b$ and $\rho x\notin\Z$, then
\[
 \frac{{t\choose x}{n\choose t-x}}{{n-s\choose t}}
p^{-\lfloor x\rho+xC\rho/\log n\rfloor}
\le \left(\frac{O(1)}{\log^{2/b}n}\right)^x n^{-\Theta(x)}
=n^{-\Theta(x)}.
\]
Also note that 
\begin{align*}
\sum_{x\in\{b,2b\}} \frac{{t\choose x}{n\choose t-x}}{{n-s\choose t}}
p^{-\lfloor x\rho+xC\rho/\log n\rfloor}
&\le 
(1+o(1))\sum_{x\in\{b,2b\}} 
\left(\frac{e^{C+1}B^2}{A\log^{2/b} n}
\right)^x\\
&=(1+o(1))\frac{e^{C+1}B^2}{A\log^2 n}
<\frac{1}{2B^2\log^2n},
\end{align*}
for all large $n$, provided that 
$A>2e^{C+1}B^4$.

Altogether, for any such $A$, summing over all relevant $x$
in \eqref{E_corrsumx}, 
we find that 
\eqref{E_corr} holds for all large $n$, 
and this completes the proof. 
\end{proof}

\subsection{The construction}
We are ready to prove Theorem \ref{T_UB}.

\begin{proof}[Proof of Theorem \ref{T_UB}]
Let $np^\rho=A\log^{2+2/b} n$, where $A$ is sufficiently large. 
Then, by Lemma \ref{L_Janson}, with high probability, 
for {\it every} $S\subset F$ such that 
\begin{itemize}[nosep]
\item $V(F)$ is ordered, 
\item $v(F)\le 2e(H)\log n$, and 
\item $(S,F)$ is $\alpha_*$-sparse, 
\end{itemize}
{\it every} subgraph $S'\subset\Gnp$
with $v(S')=v(S)$ has an $(S,F)$-extension  
in $\Gnp$.
By Lemma \ref{L_smallC} and \eqref{E_hatC2}, 
for all large $n$, we have 
$v(M_f^*)\le 2e(H)\log n$ for all red edges $f\in {\mathcal R}^*$
of $W^*$. 
Also, by Lemma \ref{L_sparse}, 
note that $(C_f^*,M_f^*)$ is $\alpha_*$-sparse for
all $f\in {\mathcal R}^*$. 
For the rest of this proof, we will assume that these statements  
hold. Using these, we will show how to activate an arbitrary edge $g$ 
missing from $\Gnp$. 

{\it Base step.} Consider the modified $\alpha_*$-core $M^*=M^*_{e^*}$
of $W^* =W^*_{e^*}$. Since $M^*\subseteq W^*\subseteq A^*$, it follows by 
\eqref{E_alphastar} that 
\[
\rho^{\rm max}(e^*,M^*\cup e^*)\le 
\rho^{\rm max}(e^*,A^*\cup e^*)<\rho+1/\log n
<\alpha_*. 
\]
Hence $(e^*,M^*\cup e^*)$ is $\alpha_*$-sparse, 
and so, there is an $(e^*,M^*\cup e^*)$-extension of
$g$ in $\Gnp$. In other words, we can embed
a copy of $M^*$ in $\Gnp$ based at $g$. 
We call $g$ the {\it copy}
of $e^*$ in $\Gnp$, and say that $g$
has been {\it processed}.

{\it Next step.} Recall that $M^*$ is associated with 
a copy $H^*=H^*_{e^*}$ of $H$ that $e^*$ completes. 
Specifically, 
\[
M^*=\bigcup_{f\in E(H^*\setminus e^*)}C^*_f.
\] 
Let $f_1,\ldots,f_k$ be the red edges
in $H^*\setminus e^*$ (for which $C^*_f\neq f$). 
Each $f_i$ has an {\it unprocessed} copy $g_i$ in $\Gnp$, 
and a copy $C_{g_i}$ of $C^*_{f_i}$
embedded in $\Gnp$ based at $g_i$. 
Therefore, since $(C^*_{f_i},M^*_{f_i})$ is $\alpha_*$-sparse, 
there is an $(C^*_{f_i},M^*_{f_i})$-extension of 
$C_{g_i}$ in $\Gnp$. We process the edges $g_1,\ldots,g_k$
in turn, one at a time, embedding such extensions in $\Gnp$. 
In each step, for each red edge $f$ in $H^*_{f_i}$, an unprocessed
copy $g'$ in $\Gnp$ is created. 

Let us stress that red edges in ${\mathcal R}^*$ can 
have multiple unprocessed copies in $\Gnp$, 
as a single red edge may appear in different $H^*_{f_i}$;
see Figure \ref{F_copies} for an illustrative example. 

\begin{figure}[h]
\centering
\includegraphics[scale=1]{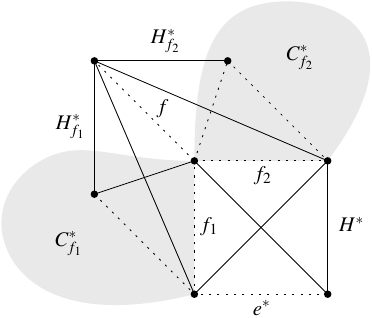}
\caption{There are two red edges
$f_1,f_2$ in the copy $H_*$ of $H$ that $e_*$ completes
in $W^*$. There is a red edge $f$  in both $H^*_{f_i}$
that is not contained in either of the $\alpha_*$-cores
$C^*_{f_i}$. In the next step of the construction, 
$(C^*_{f_i},M^*_{f_i})$-extensions of $C^*_{f_i}$ will be added to $\Gnp$. 
Since both of these extensions involve the red edge $f$, 
these extensions can place copies of $f$ at different locations in 
$\Gnp$. Each copy will need to be processed separately
later on in the construction. 
}
\label{F_copies}
\end{figure}

{\it Inductive step.} We continue recursively, until there 
is no unprocessed copy of any $f\in {\mathcal R}^*$
remaining in $\Gnp$. 
(Note that if $M^*_f$ is a black copy of $H$ minus an edge
in $W^*$, then processing a copy of $f$ in $\Gnp$ 
does not create any further edges to process.) 
In this way, 
by taking the union of all the extensions embedded 
during this recursive procedure, 
we obtain some subgraph $A_g\subseteq \Gnp$. 

{\it Activation.} 
Finally, we claim that $(g,A_g)\in{\mathcal A}(H)$, 
which completes the proof of the theorem. 
Indeed, by induction, we claim that every 
copy  
of each red edge $f\in{\mathcal R}^*$ in $\Gnp$ is activated.

For the base case, we first consider those $f$ for which 
all edges in $H^*_f\setminus f$ are 
black (i.e., the red edges that are 
activated at time $t=1$). 
Trivially, every copy of such an $f$ in $\Gnp$ is activated 
by the embedding of $H^*_f\setminus f$ in $\Gnp$, 
since clearly $M^*_f=H^*_f\setminus f$.

On the other hand, for the 
other red edges $f$ that are activated
at some time $t>1$, we recall that 
$f$ completes some copy $H_f^*$ of $H$. 
By construction, for each copy $g'$ of such an $f$ in $\Gnp$,
an extension $M^*_f$ of $f$ has been added to $\Gnp$, 
which contains a copy of $H^*_f$ based at $g'$. 
By induction, the copy in $\Gnp$
of each red edge in $H^*_f\setminus f$
is activated. 
Hence, $g'$ is also activated. 
Finally, recalling that $e^*$ is the final edge red edge in ${\mathcal R}^*$
to be activated,  
and since $g$ is the copy of $e^*$ in $\Gnp$, we see that
$g$ is activated.  
\end{proof}

\section{Hitting time results}
\label{S_HT}

In this section, to further demonstrate the 
utility of formula \eqref{E_rho} for $\rho(H)$, 
we discuss hitting time results for all $H$
with $\rho(H)\le 1$. We will 
also prove such results for all $H$ with 
a {\it leaf} (i.e., a vertex of degree 1). 
The main results are stated as Theorems \ref{thm:hitting time} 
and \ref{thm:hitting_connectivity} below.

We consider the usual random graph 
process $(\cG_m,0\le m\le{n\choose2})$, where $\cG_{m+1}$
is obtained from $\cG_m$  by adding a uniformly random edge (from amongst 
those missing from $\cG_m$). In discussing the {\it hitting time} for some graph property, we
mean the first time that $\cG_m$ has this property.

\subsection{Graphs $H$ with a leaf}
\label{S_leaf}

We recall that in \cite[Proposition 26]{BBM12}, for graphs $H$ with a leaf, it is shown that 
\begin{equation}\label{E_BBMleaf}
p_c(n,H)=\Theta(n^{-1/\beta(H)}),
\end{equation}
where
\begin{equation}\label{E_beta}
\beta(H)
=\min_{e\in E(H)}\rho^{\rm max}(\emptyset,H\setminus e)
=\min_{e\in E(H)}\max_{\emptyset\subset F\subset H\setminus e}\frac{e(F)}{v(F)}.
\end{equation}
In particular, this gives a large family of graphs $H$
for which $p_c(n,H)=O(n^{-1/\rho(H)})$ holds, as in 
Question \ref{Q_nolog} above. 

Let $\cH^-_{\beta}$ be the set of all graphs $H^-=H\setminus e$
obtained by removing an edge $e$ from $H$
for which $\rho^{\rm max}(\emptyset,H^-)=\beta(H)$
is minimal.

\begin{thm}
\label{thm:hitting time}
Suppose that $H$ has a leaf. Then, with high probability, 
the hitting time for $H$-percolation
in the 
random graph 
process $(\cG_m,0\le m\le{n\choose2})$
coincides with the first time that $\cG_m$
contains a copy of some $H^-\in \cH^-_{\beta}$.
\end{thm}

The main idea is to note that the first
copy $H^-$ of $H$ minus an edge
to emerge in the random graph process will, with high probability, belong to $\cH_\beta^-$. 
In the following proof, we 
consider some vertex $x$ that is neighbors with a leaf, and take cases with respect to 
whether 
$H^-$ has subgraphs $F$ with density
$e(F)/v(F)=\beta(H)$ that contain $x$. 

\begin{proof}
Suppose that $H$ has a leaf.  

Recall that $\cH^-_\beta$ is the set
of all $H^-=H\setminus e$, with $e\in E(H)$, such that $\rho^{\rm max}(\emptyset,H^-)$
attains the minimal possible value $\beta=\beta(H)$, 
as in \eqref{E_beta} above. 
We also let $\cH^-$ denote the set of all $H^-=H\setminus e$, with $e\in E(H)$. 
Recall (see, e.g., \cite[Theorem 3.4]{JLR00}) that the hitting time $\tau^-$, 
in the 
random graph 
process $(\cG_m,0\le m\le{n\choose2})$, 
for the appearance of a copy of some graph in 
$\cH^-$ coincides with the hitting time $\tau^-_\beta$ for the appearance of a copy of some graph
in $\cH^-_\beta\subseteq \cH^-$. 
Of course, if $G\neq\langle G\rangle_H$, then $G$ must contain at least one copy of 
some $H^-\in\cH^-$. 
As such, it remains to prove that, with high probability, $\langle \cG_{\tau^-}\rangle_H=K_n$. 
To this end, we claim that $\langle G\rangle_H=K_n$
for {\it any} graph $G\subseteq K_n$ that contains:  
\begin{itemize}[nosep]
\item  some $H^-_\beta$ that is a copy of a graph in $\cH^-_\beta$; 
\item  at least $2v(H)$ disjoint copies of every graph $F$ on at most $v(H)$ vertices 
with $\rho^{\rm max}(\emptyset,F)<\beta$; 
\item at least $2v(H)$ disjoint $(H_1,H_2)$-extensions of 
every set of vertices $U$ of size $v(H_1)$ for every 
$\beta$-sparse pair $(H_1,H_2)$ with $v(H_2)\le v(H)$.
\end{itemize}
We note that, with high probability, $\cG_{\tau^-}$ has all of these properties; 
see \cite[Remark 3.7]{JLR00},
\cite[Theorem 3]{Spe90}, 
and Lemma~\ref{L_Janson} above. 
Hence the claim implies the theorem. 

In proving the claim, we take three cases. 
Let $x$ be the neighbor of some leaf in $H$. 
In what follows, we say that a subgraph $F\subseteq H^-_\beta$
is $\beta$-dense if $e(F)/v(F)=\beta$. 

{\it Case 1.} Suppose that no $\beta$-dense subgraph 
$F\subseteq H^-_\beta$ contains $x$. 
Let $H^*$ be an inclusion-maximal $\beta$-dense subgraph of 
$H^-_\beta$. 
Along similar lines as Lemma~\ref{L_sparse} above, 
it can be seen that the pair $(H^*,H^-_\beta)$ is $\beta$-sparse. 
Hence, we can find  $v(H)-2$ disjoint 
$(H^*,H^-_\beta)$-extensions of $H^*$. 
Let $x_1,\ldots,x_{v(H)-2}$ be the copies of $x$ in these extensions. 
To see that $\langle G\rangle_H=K_n$, 
note that we can activate:  
\begin{itemize}[nosep]
\item the missing edge $e$ in $H^-_\beta$; 
\item the copy of $e$ in each $(H^*,H^-_\beta)$-extension of $H^*$, 
if the endpoints of $e$ 
are not both in $H^*$ (otherwise there is no such copy); 
\item all edges that contain one of the vertices $x_1,\ldots,x_{v(H)-2}$;
\item and then finally any other missing edge $\{u,v\}$, 
by completing a copy of $H$ with vertices $u,v,x_1,\ldots,x_{v(H)-2}$.
\end{itemize}

In the next two cases, we assume there is a $\beta$-dense subgraph
of $H^-_\beta$ that contains $x$, and let 
$H^*_x$ be an inclusion-minimal subgraph with these properties. 

{\it Case 2.} 
Suppose that all  $\beta$-dense subgraphs
of $H^-_\beta$ contain $x$. 
Then all such subgraphs contain $H^*_x$. 
Hence, for any vertex $z\neq x$ in $H^*_x$, 
we have that $\rho^{\rm max}(\emptyset,H^-_\beta\setminus z)<\beta$. 
Therefore, we can find $v(H)-2$ disjoint copies of $H^-_\beta\setminus z$ 
in $G$ 
that are also disjoint from 
$H^-_\beta$. 
Let $x_1,\ldots,x_{v(H)-2}$ be the copies of $x$ in these subgraphs.
To see that $\langle G\rangle_H=K_n$, 
we activate: 
\begin{itemize}[nosep]
\item the missing edge $e$ in $H^-_\beta$; 
\item all edges that contain $x$ (and note that $x$ may now 
play the role of the missing vertex $z$ in the $v(H)-2$ copies of $H^-_\beta\setminus z$); 
\item the copy of $e$ in every extended (with $x$ playing the role of $z$) copy of $H^-_\beta\setminus z$, 
if $e$ is not incident to $x$; 
\item all edges that contain one of the vertices $x_1,\ldots,x_{v(H)-2}$; 
\item and then finally any missing edge $\{u,v\}$ by completing a copy of $H$ 
with vertices $u,v,x_1,\ldots,x_{v(H)-2}$.
\end{itemize}

{\it Case 3}. Finally, we  assume that, while $H^-_\beta$ has a $\beta$-dense 
subgraph that contains $x$, not all such subgraphs contain $x$. 
In this case, we will use the fact that 
an intersection of any $\beta$-dense subgraph $F\subseteq H^-_\beta$ 
with $H^*_x$ is either empty or $\beta$-dense. 
This follows by the next claim 
(taking $A=H^-_\beta$, $A_1=F$ and $A_2=H^*_x$). 

\begin{claim}
Suppose that $\rho^{\rm max}(\emptyset,A)=\beta$. 
Let $A_1,A_2$ be $\beta$-dense subgraphs of $A$ such that 
$A_1\cap A_2\neq\emptyset$. Then $A_1\cap A_2$ is $\beta$-dense. 
\end{claim}

\begin{proof}
Since $A_1,A_2$ are $\beta$-dense, it follows that 
\[
e(A_1\cup A_2)
\ge\beta(v(A_1)+v(A_2))-e(A_1\cap A_2).
\]
Therefore, since $A_1\cup A_2$ is $\beta$-dense, 
it follows that 
$e(A_1\cap A_2)=\beta(v(A_1\cap A_2))$, 
and so $A_1\cap A_2$ is $\beta$-dense, as claimed. 
\end{proof} 

If there is no $\beta$-dense subgraph of $H^-_\beta$
that does not contain $x$ and have a non-empty intersection
with $H^*_x$, then we put $H^*=\emptyset$. 
Otherwise, let $H^*$ be 
an inclusion-maximal
$\beta$-dense subgraph of $H^*_x$ 
that does not contain $x$. 
Let $H^*_-$ be an inclusion-minimal $\beta$-dense subgraph of $H^*$ 
(where $H^*_-=\emptyset$ if $H^*=\emptyset$). 
Also, let $H^*_+$ be an inclusion-maximal $\beta$-dense subgraph 
of $H^*\setminus H^*_-$, if such a subgraph exists;
otherwise, if $\rho^{\rm max}(\emptyset,H^*\setminus H^*_-)<\beta$, we put 
$H^*_+=\emptyset$. 

Next, we let $Y$ be the (possibly empty) union 
of all $\beta$-dense subgraphs of $H^-_\beta$ that do not contain $x$
and are vertex-disjoint with $H^*$. 
Note that $Y\cap H^*_x=\emptyset$, and 
in particular $x\notin Y$. 
There are no edges between $Y$ and $H^*_x$, as otherwise
the density of $Y\cup H^*_x$ would be strictly greater than $\beta$.

We claim that every $\beta$-dense subgraph of $H^-_\beta$ 
that contains $Y\cup H^*_+$ as a {\it proper} 
subgraph must contain all of $H^*_-$. 
Indeed, if such a subgraph: 
\begin{itemize}[nosep]
\item contains $x$ but not all of $H^*_-$, then this would contradict the minimality of $H^*_x$;
\item does not contain $x$, and some but not all of $H^*_-$, then this would 
contradict the minimality of $H^*_-$;
\item does not contain $x$, or any of $H^*_-$, then this would contradict 
the maximality of $H^*_+$ (as then, by the choice of $Y$ 
and since $Y\cup H^*_+$ is a proper subset, it would contain some other vertex 
$H^*\setminus H^*_-$ that is not already in $H^*_+$). 
\end{itemize}

To conclude, take any vertex $z$ in 
$H^*_-$ if $H^*\neq\emptyset$, and take any $z\neq x$ in $H^*_x$ otherwise. 
 The pair $(Y\cup H^*_+,H^-_\beta\setminus z)$ is $\beta$-sparse, 
so we can find $v(H)-2$ disjoint $(Y\cup H^*_+,H_\beta^-\setminus z)$-extensions 
of $Y\cup H^*_+$ 
in $G$ that also disjoint from $H^-$. 
Since $z$ does not have neighbors in $Y\cup H^*_+$, 
each such extension can be extended to a copy of $H^-$ by replacing $z$ with $x$, 
as soon as all edges between $x$ and vertices outside of $H^-$
have been activated. 
Hence, using the copies 
$x_1,\ldots,x_{v(H)-2}$ 
of $x$ in these $v(H)-2$ resulting copies of $H^-$, 
the remaining edges of $K_n$ can be activated in the same way as before. 

The proof is complete. 
\end{proof}

\subsection{A characterization}

Our next result characterizes
all graphs $H$ with $\rho(H)\le 1$. Recall (as discussed in Section \ref{S_Intro})
that we always assume 
that $H$ has no isolated vertices. 

\begin{lemma}
If $H$ has at least two cycles then $\rho(H)\ge 1$. Moreover, $\rho(H)=1$ if and only if 
$H$ 
\begin{itemize}[nosep]
\item is a cycle, or  
\item has at least two cycles and some edge $e$ for which 
each connected component of 
$H\setminus e$ has at most one cycle. 
\end{itemize}
\label{lm:rho=1}
\end{lemma}

\begin{proof}
If $H$ has at least two cycles, then no edge of $H$ is in all of
its cycles. Hence, for any $(e,A)\in\mathcal{A}(H)$, 
there is a cycle in $A$, and so $\rho(H)\ge 1$. 

Next, let us assume $\rho(H)=1$ and that $H$ is not a cycle. 

Towards a contradiction, let us assume that  
$H$ has only one cycle but is not itself a cycle. Then 
there is some vertex $x$ that is neighbors with a leaf in $H$. 
To obtain a contradiction with our assumption that $\rho(H)=1$, 
we will construct a pair $(f,A)\in\mathcal{A}(H)$
for which $\rho^{\rm max}(f,A\cup f)<1$ as follows: 
Obtain $H^-$ from $H$ by deleting an edge $f$ in its cycle. 
We obtain $A$ by taking a disjoint union of copies $H_1^-,\ldots,H_{v(H)-2}^-$
of $H^-$ together with two isolated vertices $y,z$. Let $f$ be the missing 
edge $\{y,z\}$. Clearly, $\rho^{\rm max}(f,A\cup f)<1$. 
To see that $A$ activates $f$, 
we first activate the copies $f_i$ 
of the missing edge $f$ in each of the $H_i^-$.
Let $x_i$ be the copies of $x$ in the $H_i^-$, and note that, by the choice of $x$, 
every missing edge
containing some $x_i$ can be activated. 
In particular, all edges of the clique on $x_1,\ldots,x_{v(H)-2},y,z$
except $f$ can be activated in this way; and then $f$ 
can also be activated. 

Towards another contradiction,  
suppose that $H$ does not have an edge whose 
deletion leaves at most one cycle in every connected component. 
Then, for every $(f,A)\in\mathcal{A}(H)$, there exists a connected subgraph 
$F\subseteq A$ on at most $v(H)$ vertices with at least two cycles, 
and hence $\rho^{\rm max}(f,A\cup f)\ge 1+1/v(H)$, in contradiction 
to out assumption that $\rho(H)=1$. 

Finally, let us assume that $H$ is either a cycle or has at least two cycles 
and an edge $e$ whose deletion leaves at most one cycle in every connected component. 
It remains only to prove that $\rho(H)=1$.

{\it Case 1.} The case that $H=C_m$ is a cycle is the simplest. 
For instance, as noted in 
\cite[Proposition 24]{BBM12}, if we let $A$ be the union of a triangle
and a path of length at least $m$ that share exactly one vertex, then 
$\langle A\rangle_H$ is the clique on $V(A)$. 
(Other constructions are possible, 
such as a long path and an extra edge between a pair of vertices at distance $m-1$.)
In particular, 
the edge $f$ between the endpoints 
of the path is activated. Since the path can be arbitrarily long, 
this shows that $\rho(C_m)=1$. 

{\it Case 2.} In the second case, $H$ is a disjoint union of 
a graph $H_2$ that has at most one cycle in each of its connected components, 
and a connected graph $H_1$ that contains either:
\begin{itemize}[nosep]
\item a single cycle, 
\item two cycles that share exactly one vertex,
\item two disjoint cycles connected by a path, or 
\item three paths that share their endpoints (i.e., two cycles that share a path). 
\end{itemize}

{\it Case 2a.} If $H_1$ is a cycle
then we can argue along the same lines as in Case 1, 
taking $A$ to be the disjoint union of a copy of $H_2$ together with a
triangle attached to an arbitrarily long path, as before. If $H_1$ contains
a single cycle, we proceed in the same way, but attach 
large enough perfect trees to all vertices along the path and triangle. 

In the remaining cases, for simplicity, let us assume that $H=H_1$,  
since otherwise, in all of our constructions of $A$, we can simply add a disjoint 
copy of $H_2$ to $A$. Likewise, let us assume that the minimum degree
$\delta(H)=2$, since otherwise we can simply add perfect trees to all vertices in $A$. 

{\it Case 2b.} Suppose that $H$ is the union of two cycles, of lengths $\ell$ and $m$, 
that share a single vertex $x$. Then to obtain $A$, take a copy of $C_m$
and attach an arbitrarily long path of length $(\ell-1)+k(\ell-2)$, for some $k\ge0$. 
By induction, the edge $f$ between the endpoints of the path is activated. 
Taking $k$ arbitrarily large, we see that $\rho(H)=1$ in this case. 

{\it Case 2c.} If $H$ is the union of two cycles, of lengths $\ell$ and $m$, 
that are connected by a path of length $p$, 
the argument is essentially the same as in Case 2b. We obtain $A$, we take a copy of 
$C_m$, attach a path of length $p$ from $C_m$ to some other vertex $x$, and then 
attach a path of length  $(\ell-1)+k(\ell-2)$ to $x$. 
By induction, the edge $f$ from $x$ to the end of
the path is activated. 

{\it Case 2d.} Finally, suppose that $H$ is the union of three paths, 
of lengths $\ell_1$, $\ell_2$ and $\ell_3$, that share their endpoints. 
To construct $A$, we start with two paths of lengths $\ell_1$ and $\ell_2$
that share their endpoints $x$ and $y$. Next, we attach a path $P_x$ to $x$ 
of length $k(\ell_2+\ell_3-2)+(\ell_2-1)$, and a path $P_y$ to $y$ of length $2k(\ell_1-1)$. 
These paths are otherwise disjoint with the two paths between $x$ and $y$,
and are disjoint with each other. Let $f$ be the missing edge between the endpoints between 
these paths. 
Note that the edge between $y$ and the vertex along $P_x$ at distance
$\ell_3-1$ from $x$ can be activated. Thereafter, the edge between $x$ and the vertex along $P_y$
at distance $\ell_1-1$ from $y$ can be activated. In this way, 
by induction, it can be seen that $f$ 
is eventually activated. 
\end{proof}

\subsection{Graphs $H$ with $\rho<1$}
It follows immediately 
by Lemma \ref{lm:rho=1}
that $\rho(H)<1$ if and only if $H$ has at most one cycle and a leaf. 
For such a graph $H$ and an edge $e\in E(H)$, 
let $f(e,H)$ be the maximum number of edges in a tree in $H\setminus e$. 
Let 
$f(H)=\min_ef(e,H)$, 
minimizing over all $e\in E(H)$ for which $H\setminus e$ is a forest. 
Then, it can be seen that 
$p_c(n,H)=\Theta(n^{-(f+1)/f})$. 
Indeed, when $p\ll n^{-(f+1)/f}$, with high probability $\Gnp$ has no 
trees of size $f$ (see, e.g., \cite[Theorem 3.4]{JLR00}), and therefore no edge can be activated. 
On the other hand, when $p\gg n^{-(f+1)/f}$, with high probability $\Gnp$ 
has many copies of any tree with at most $f$ edges 
(see, e.g., \cite[Remark 3.7]{JLR00}), and this is enough to activate all of $K_n$.

In fact, by Theorem~\ref{thm:hitting time}
we obtain the following hitting time result. 
Let $\mathcal{F}_H$ be the family of all forests $F$ with the following property:
For each $F\in \mathcal{F}_H$, there is some $e\in E(H)$ such that $F$ consists 
of trees of size $f(H)$ in the forest $H\setminus e$. Then, by 
Theorem~\ref{thm:hitting time}, 
the hitting time for $H$-percolation 
coincides with the appearance of a copy of some 
forest $F\in \mathcal{F}_H$.

\subsection{Graphs $H$ with $\rho=1$}
\label{S_rho1}

We note that, if $\rho(H)=1$ and $H$ has a leaf, then the hitting time result  
Theorem \ref{thm:hitting time} implies that the 
$H$-percolation threshold is coarse. 
Our next result shows that the threshold is sharp for all other $H$ with $\rho(H)=1$. 

\begin{thm}
\label{thm:hitting_connectivity}
Let $H$ be a graph with $\rho(H)=1$ and minimum degree $\delta(H)=2$. 
Then, with high probability, the 
hitting time  
in the 
random graph 
process $(\cG_m,0\le m\le{n\choose2})$
for $H$-percolation coincides with that of connectivity. 
\end{thm}

We note that this result
extends the result in \cite[Proposition 24]{BBM12} 
for the cases $H=C_k$, $k\ge3$, and $H=K_{2,3}$.  

\begin{proof}
If $H$ has minimum degree $\delta(H)=2$ then, according  
to Lemma~\ref{lm:rho=1}, $H$ consists of connected components 
$H_1,\ldots,H_k$, where $H_2,\ldots,H_k$ are cycles, and $H_1$ 
is either a cycle, or a union of two cycles 
that are either connected by a path or share a path 
(see the proof of Lemma~\ref{lm:rho=1}). 

Let $\cC$ be the property of being connected. 
Recall that, with high probability, the hitting time $\tau_\cC$ for connectivity 
coincides with that of the disappearance of isolated vertices, 
and $\tau_\cC=\frac{\ln n+O_p(1)}{n}\cdot{n\choose 2}$; see \cite[Theorems 7.4, 7.7]{Bol01}. 
Since $\delta(H)=2$, note that $\langle G\rangle_H\neq K_n$ for any $G\subset K$
with an isolated vertex. 
Therefore, it suffices to find a graph property $\cP$ such that: 
\begin{itemize}
\item $\cC\cap\cP$ is increasing, 
\item with high probability $\tau_\cC=\tau_{\cC\cap\cP}$, and 
\item deterministically any graph with $\cC\cap\cP$ is $H$-percolating. 
\end{itemize}

Recall that, in the proof of Lemma~\ref{lm:rho=1}, we constructed a pair 
$(f,A)\in\cA(H)$, where $A$ is a union of a cycle and one or two long paths,
and possibly also a disjoint copy of some graph $H_2$ that plays no essential role
(and note that perfect trees are note required in constructing $A$ here, 
as we currently assume that $\delta(H)=2$.)
We let $\cP$ be the property that every pair of non-isolated vertices has an 
$(f,A)$-extension, for some fixed $v(A)=\Theta(\log n)$. 

To show that with high probability $\tau_\cC=\tau_{\cC\cap\cP}$, 
consider $\Gnp$ with let $p=(1\pm o(1))\log n/n$. 
With high probability, every non-isolated vertex in $\Gnp$ has a neighor
of degree at least $0.5 \log n$. 
Take a pair of vertices $v_1,v_2$, and reveal their neighborhoods. 
If at least one of them is isolated, there is nothing to prove, 
so assume both vertices have neighbors $u_1,u_2$ with degrees at least $0.5 \log n$. 
Let $X$ be the number of $(f,A)$-extensions of $\{v_1,v_2\}$. 
The expected number of such extensions equals 
$\Theta(n^{v(A)}p^{v(A)+1})=n^{\omega(1)}$. 
By Janson's inequality, $\mathbb{P}(X=0)\le \exp(-\mu^2/(2\Delta))$, 
where 
\[
\Delta\le \max\{\mu,2\mu^2/ np\log n,\mu^2/(np)^3\}
\le \mu^2/0.4\log^2 n. 
\]
Therefore, $\mathbb{P}(X=0)\le e^{-0.4\log^2 n}=n^{-\omega(1)}$,
and so by taking a union bound over all pairs $\{v_1,v_2\}$ 
we complete the proof.
\end{proof}

\providecommand{\bysame}{\leavevmode\hbox to3em{\hrulefill}\thinspace}
\providecommand{\MR}{\relax\ifhmode\unskip\space\fi MR }
\providecommand{\MRhref}[2]{%
  \href{http://www.ams.org/mathscinet-getitem?mr=#1}{#2}
}
\providecommand{\href}[2]{#2}

\end{document}